\documentclass[12pt,a4paper,reqno]{amsart}
\usepackage{adjustbox}
\usepackage{amsmath}
\usepackage{amssymb}
\usepackage{amsfonts}
\usepackage{amsxtra}
\usepackage{bm}
\usepackage{booktabs}
\usepackage{caption}
\usepackage{enumerate}
\usepackage{enumitem}
\usepackage{float}
\usepackage{geometry}
\usepackage{graphicx}
\usepackage[unicode,psdextra]{hyperref}
\usepackage{interval}
\usepackage{mathrsfs}
\usepackage{mathtools}
\usepackage[final]{microtype}
\usepackage{refcount}
\usepackage{tikz-cd}
\usepackage{xspace}
\usepackage{physics}
\usepackage[backend=biber,maxbibnames=99,maxalphanames=99,style=ext-alphabetic,date=year]{biblatex}
\DeclareFieldFormat[article]{volume}{\mkbibbold{#1}}
\AtEveryBibitem{\clearfield{number}}
\intervalconfig{soft open fences}
\numberwithin{equation}{section}

\usepackage[capitalise,nameinlink]{cleveref}
\usepackage{amsthm}
\theoremstyle{plain}
\newtheorem{theorem}{Theorem}[section]
\newtheorem{lemma}[theorem]{Lemma}
\newtheorem{corollary}[theorem]{Corollary}
\newtheorem{proposition}[theorem]{Proposition}
\theoremstyle{definition}
\newtheorem{definition}[theorem]{Definition}

\newtheorem{remark}[theorem]{Remark}

\numberwithin{equation}{section}

\numberwithin{equation}{section}

\newtheorem{conjecture}[theorem]{Conjecture}
\newcommand{\cat}[1]{\mathbf{#1}}
\DeclareMathOperator{\Bl}{Bl}
\DeclareMathOperator{\Coh}{\cat{Coh}}

\DeclareMathOperator{\Hom}{Hom}
\DeclareMathOperator{\id}{Id}
\DeclareMathOperator{\Per}{\cat{Per}}
\DeclareMathOperator{\Spec}{Spec}
\DeclareMathOperator{\cEnd}{\mathcal{E}nd}
\DeclareMathOperator{\cHom}{\mathcal{H}om}
\DeclareMathOperator{\ch}{ch}
\DeclareMathOperator{\coker}{coker}
\DeclareMathOperator{\cone}{cone}
\DeclareMathOperator{\ext}{ext}
\newcommand{\cA}{\mathcal{A}}
\newcommand{\cB}{\mathcal{B}}
\newcommand{\cC}{\mathcal{C}}
\newcommand{\cD}{\mathcal{D}}
\newcommand{\cE}{\mathcal{E}}
\newcommand{\cF}{\mathcal{F}}
\newcommand{\cG}{\mathcal{G}}
\newcommand{\cH}{\mathcal{H}}

\newcommand{\cN}{\mathcal{N}}
\newcommand{\cO}{\mathcal{O}}

\newcommand{\cT}{\mathcal{T}}
\newcommand{\cV}{\mathcal{V}}

\newcommand{\CC}{\mathbb{C}}

\newcommand{\PP}{\mathbb{P}}
\newcommand{\QQ}{\mathbb{Q}}
\newcommand{\RR}{\mathbb{R}}
\newcommand{\ZZ}{\mathbb{Z}}

\begin{document}
\title{Stability conditions on blowups}
\author{Nantao Zhang}
\address{Department of Mathematical Sciences, Tsinghua University, 100084 Beijing, China}
\email{znt21@mails.tsinghua.edu.cn}

\date{}

\begin{abstract}
  We study the relation between perverse stability conditions and geometric stability conditions under blow up. We confirm a conjecture of Toda in some special cases and show that geometric stability conditions can be induced from perverse stability conditions from semiorthogonal decompositions associated to blowups.
\end{abstract}

\maketitle
\tableofcontents

\section{Introduction}
\label{sec:intro}

Bridgeland stability conditions on triangulated categories introduced by Bridgeland in \cite{Bridgeland2007} are a powerful tool for the study of moduli spaces of stable sheaves. Moreover, they are connected to birational geometry as suggested in \cite{Toda2013b} and more recently from a different perspective in \cite{Halpern-Leistner2024,HJR2024}. In this paper, we focus on a special kind of birational map, the blow up \(f: X = \Bl_{C}Y \to Y\), and try to explore the relationship of the stability conditions on \(X\) and on \(Y\). In \cite{Bridgeland2008,AB2012,BMT2014}, the authors gave a systematic method to construct stability condition on surfaces and threefolds (conjecturally) by tilting from abelian category of coherent sheaves, which we call geometric stability conditions. In \cite{Toda2013b}, the author constructed a new family of Bridgeland stability condition associated to extremal contraction (also conjecturally in threefold case) also by tilting but from the abelian category of perverse coherent sheaves, which we call perverse stability conditions.

On the other hand, there is a standard semiorthogonal decomposition associated to blow up of a smooth center
\[D^{b}\Coh(X) = \langle Lf^{*}D^{b}\Coh(Y), \Phi_{0}D^{b}\Coh(C, \cdots, \Phi_{c - 2}D^{b}\Coh(C))\rangle\]
where \(\Phi_{i}\) are fully faithful admissible embeddings and \(c\) is the codimension of \(C\) in \(Y\). In \cite{BLM+2023}, the authors proposed a general method to induce stability conditions satisfying some specific conditions from a whole triangulated category to its semiorthogonal components complementary to exceptional collections. In this paper, we apply their construction to the above semiorthogonal decomposition in the case of of blowing up a point, and derive a similar approach to deal with the case of blowing up a curve. Our main result is the followings,
\begin{theorem}(\cref{thm:2}, \cref{thm:4})
  The perverse stability condition on \(D^{b}\Coh(X)\) induces geometric stability condition on \(D^{b}\Coh(Y)\).
\end{theorem}

Along the way, we show the following result.
\begin{theorem}(\cref{thm:3})
  The generalized Bogomolov inequality for perverse tilted stability of \(X\) in the threefold case can be induced from the generalized Bogomolov inequality for geometric tilted stability of \(Y\).
\end{theorem}
The generalized Bogomolov inequality for geometric tilted stability has been confirmed in many cases, for example, Fano varieties of Picard rank 1 in \cite{Li2018} (with \(\PP^{3}\) worked out earlier in \cite{Macri2014}), abelian threefolds \cite{BMS2016}, and quintic threefolds \cite{Li2019}. Therefore, our results confirm Toda's conjecture in these cases. It is worth noting, that the generalized Bogomolov inequality in this paper refers to the original one introduced in \cite{BMT2014}, which has counterexamples, for example \cite{Schmidt2017}. So our result can not apply directly to variants of generalized Bogomolov inequalities for example in \cite{BMS+2017}.

\subsubsection*{Plan of the paper}
In \cref{sec:revi-tilt-bridg}, we recall some background on Bridgeland stability conditions and the tilting process to construct stability condition. In \cref{sec:induc-t-struct}, we recall semiorthogonal decompositions and explain how to induce stability conditions on semiorthogonal components following \cite{BLM+2023}. In \cref{sec:surface-case}, we discuss the relation between perverse stability condition and geometric stability condition in the surface case. And in~\cref{sec:threefold-case-1}, we extend our results to threefold case.

\subsubsection*{Conventions}
In this paper, we work over \(\CC\) or an algebraic closed field \(k\) of characteristic zero. For morphism \(f: X \to Y\), we use \(f_{*}, f^{*}\) to denote the underived pushforward and pullback and use \(Lf^{*}, Rf_{*}\) to denote the derived version. Also, we distinguish \(\langle \bullet \rangle\) the smallest triangulated closure in triangulated category and \(\langle \bullet \rangle_{\ext}\) generation only by extension in triangulated category.

\subsubsection*{Acknowledgements}

The author is very grateful to his supervisor Will Donovan to continuous encouragement and helpful discussions during the work and to Yukinobu Toda for hospitality during my visit in Kavli IPMU and many useful discussions on the topic. The author would also like to thank Linpu Gao for valuable help in discussion on bounded t-structures. The author is also indebted to Dongjian Wu and Tianle Mao for various related conversations. This research is supported by Yau Mathematical Science Center.

\section{Review on tilting and Bridgeland stability}
\label{sec:revi-tilt-bridg}

We begin by a quick review on tilting and some weak form of Bridgeland stability conditions. We follow the notation of \cite{BLM+2023}. Let \(\cD\) be a triangulated category and let \(K(\cD)\) denote the Grothendieck group of \(\cD\). Fix a finite rank lattice \(\Lambda\) and a surjective homomorphism \(v: K(\cD) \twoheadrightarrow \Lambda\).

\subsection*{T-structures}

\begin{definition}(\cite[Section 1.3]{BBD+2018}) A \textit{t-structure} on a triangulated category \(\cD\) consists of a pair of two full subcategories \((\cD^{\leq 0}, \cD^{\geq 0})\) such that
  \begin{enumerate}
    \item For \(X\) in \(\cD^{\leq 0}\) and \(Y \in \cD^{\geq 1}\), we have \(\Hom(X, Y) = 0\).
    \item We have \(\cD^{\leq 0} \subset \cD^{\leq 1}\) and \(\cD^{\geq 0} \supset \cD^{\geq 1}\).
    \item For every \(X\) in \(\cD\), there exists a distinguished triangle
    \[A \to X \to B \xrightarrow{ [1]}\]
    such that \(A \in \cD^{\leq 0}\) and \(B \in \cD^{\geq 1}\).
  \end{enumerate}
  Here \(\cD^{\leq n}:= \cD^{\leq 0}[-n], \cD^{\geq n}:= \cD^{\geq 0}[-n]\). The t-structure is called bounded if for every \(E \in \cD\), we have \(E \in \cD^{\geq m_{1}}\) and \(E \in \cD^{\leq m_{2}}\) for some \(m_{1}, m_{2} \in \ZZ\).

  If \((\cD^{\leq 0}, \cD^{\geq 0})\) is a t-structure, then \(\cA := \cD^{\leq 0} \cap \cD^{\geq 0}\) is an abelian category called \textit{heart} \cite[Théorème 1.3.6]{BBD+2018}.

  The inclusion map \(\cD^{\leq n} \to \cD\) admits a right adjoint \(\tau_{\leq n}\) and \(\cD^{\geq n} \to \cD\) admits a left adjoint \(\tau_{\geq n}\). And we denote by \(H_{\cA}^{i} := \tau_{\geq i} \tau_{\leq i} : \cD \to \cA\) the cohomological functor with respect to t-structure \cite[Proposition 1.3.3(i), Théorème 1.3.6]{BBD+2018}.
\end{definition}

Now we recall some useful facts about t-structures.

\begin{lemma}(\cite[Lemma 3.2]{Bridgeland2007})
  \label{lem:1}
  Let \(\cA \subset \cD\) be a full additive subcategory of a triangulated category \(\cD\). Then \(\cA\) is the heart of a bounded t-structure on \(\cD\) if and only following two conditions hold
  \begin{enumerate}
    \item for \(E, F \in \cA\) and \(k < 0\), we have \(\Hom_{\cD}(E, F[k]) = 0\), and
    \item for every \(E \in \cD\), there exists a sequence of morphisms
    \[0 = E_{0} \xrightarrow{\phi_{1}} E_{1} \to \cdots \xrightarrow{\phi_{m}} E_{m} = E\]
    such that \(\cone(\phi_{i}) \in \cA[k_{i}]\) for \(k_{1} > k_{2} > \cdots > k_{m}\).
  \end{enumerate}
\end{lemma}

In above sequence, we have \(H^{-k_{i}}_{\cA}(E) = \cone(\phi_{i})\).

\begin{lemma}
  \label{lem:2}
  Let \(\cA_{1}, \cA_{2} \subset \cD\) two hearts of bounded t-structures on \(\cD\). If \(\cA_{1} \subset \cA_{2}\), then we have \(\cA_{1} = \cA_{2}\).
\end{lemma}

\begin{proof}
  Since \(\cA_{1} \subset \cA_{2}\), we have that \(\cA_{1}[k] \subset \cA_{2}[k]\) and \(\cA_{1}^{ [1, k]} \subset \cA_{2}^{ [1, k]}\). Since \(\cA_{1}\) and \(\cA_{2}\) are bounded, we have \(\cA_{1}^{ \geq 1} \subset \cA_{2}^{\geq 1}\). Similarly, we have \(\cA_{1}^{\leq 0} \subset \cA_{2}^{\leq 0}\). But we know that \(\cA_{1}^{\geq 1} = {}^{\perp}(\cA_{1}^{\leq 0})\) \cite[Remark 1.2.1.3]{Lurie2017}. So we must have \(\cA_{1}^{\geq 1} = \cA_{2}^{\geq 1}\), \(\cA_{1}^{\geq 0} = \cA_{2}^{\geq 0}\) and therefore \(\cA_{1} = \cA_{2}\).
\end{proof}

To avoid confusion, we explain some notations about heart and triangulated categories. Given \(\cA\) the heart of bounded t-structure on \(\cD\) and a morphism \(\phi: E \to F \in \cA\), we have \(\ker \phi := H^{-1}_{\cA}(\cone(\phi))\) and \(\coker \phi := H^{0}_{\cA}(\cone(\phi))\). Furthermore, we write \(E \subset F\) if there is a morphism \(\phi: E \to F\) such that \(\ker \phi = 0\) and write \(F/E = \coker \phi\). For \(E,F \subset G\), we write \(E \cap G := \ker (E \oplus F \to G)\).

\subsection*{Bridgeland stability conditions}

A (weak) Bridgeland stability condition is a pair of  a bounded t-structure and a (weak) stability function.
\begin{definition}
  Let \(\cA\) be an abelian category. Then we call a group homomorphism \(Z: K(\cA) \to \CC\) a \textit{weak stability function} on \(\cA\) if for every \(E \in \cA\), we have \(\Im Z(E) \geq 0\), and if \(\Im Z(E) = 0\), we have \(\Re Z(E) \leq 0\). If moreover, for every \(0 \neq E \in \cA\) such that \(\Im Z(E) = 0\), we have \(\Re Z(E) < 0\), we say that \(Z\) is a \textit{stability function} on \(\cA\).

  Given a weak stability function \(Z\), we have a \textit{slope} \(\mu\) for every \(E \in \cA\)
  \[\mu_{Z}(E) :=
    \begin{cases}
      - \frac{\Re Z(E)}{\Im Z(E)} & \text{if } \Im Z(E) > 0 \\
      + \infty & \text{otherwise} \\
    \end{cases}
  \]
  An object \(0 \neq E \in \cA\) is \(\mu_{Z}\)-\textit{semistable} if for every proper subobject \(F\), we have \(\mu_{\sigma}(F) \leq \mu_{\sigma}(E)\). And if moreover the inequality is strict, we call \(E\) \(\mu_{Z}\)-\textit{stable}.
\end{definition}

\begin{definition}
  A \textit{weak stability condition} on \(\cD\) is a pair \(\sigma = (\cA, Z)\) where \(\cA\) is the heart of a bounded t-structure on \(\cD\) and \(Z: \Lambda \to \CC\) is a group homomorphism, such that
  \begin{enumerate}
    \item The composition \(K(\cA) = K(\cD) \xrightarrow{v} \Lambda \xrightarrow{Z} \CC\) is a weak stability function on \(\cA\). And \(E \in \cA\) is called \(\sigma\)-(semi)stable if it is \(\mu_{Z}\)-(semi)stable. We write write \(\mu_{\sigma} := \mu_{Z}\).
    \item (HN filtration) We require any object \(E\) of \(\cA\) to have a \textit{Harder-Narasimhan filtration} (HN filtration in short), that is an exact sequence in \(\cA\)
    \[0 = E_{0} \to E_{1} \to \cdots \to E_{m} = E\]
    such that \(F_{i} := E_{i}/E_{i - 1}\) is \(\sigma\)-semistable and \(\mu_{\sigma}(F_{1}) > \mu_{\sigma}(F_{2}) > \cdots > \mu_{\sigma}(F_{m})\). And we define \(\phi_{\sigma}^{+}(E) = \mu_{\sigma}(F_{1})\) and \(\phi_{\sigma}^{-}(E) = \mu_{\sigma}(F_{m})\).
    \item (Support condition) There exists a constant \(C > 0\) such that for every \(\sigma\)-semistable \(0 \neq E \in \cA\), we have
    \[\frac{\abs{Z(E)}}{\norm{v([E])}} \geq C\]
    where \(\norm{\cdot}\) is some fixed norm on \(\Lambda\).
  \end{enumerate}
\end{definition}

The following equivalent description of the support condition is useful in geometric settings.
\begin{lemma}(\cite[Lemma 11.4]{BMS2016})
  \label{lem:4}
  The support condition is equivalent to the existence of a quadratic form \(Q\) on \(\Lambda_{\RR}\) such that \(Q|_{\ker Z}\) is negative definite and for any \(\sigma\)-semistable object \(E \in D^{b}(X)\), we have
  \[Q(v(E)) \geq 0.\]
\end{lemma}

\subsection*{Tilting}
Tilting is a useful method to obtain new hearts of bounded t-structure from existing ones.

\begin{definition}(\cite{HRS1996})
  Let \(\cA\) be the heart of a bounded t-structure on a triangulated category \(\cD\). A \textit{torsion pair} is a pair of subcategories \((\cT, \cF)\) in \(\cA\) such that
  \begin{enumerate}
    \item For any \(T \in \cT\), and \(F \in \cF\), we have \(\Hom(T, F) = 0\).
    \item For any \(E \in \cA\), there is an exact sequence \(0 \to T \to E \to F \to 0\) in \(\cA\) with \(T \in \cT\) and \(F \in \cF\).
  \end{enumerate}

  We define the \textit{tilted heart} \(\cA^{\dagger} = \langle \cF[1], \cT \rangle_{\ext} \subset \cD\).
\end{definition}

The following proposition justifies the name.
\begin{proposition}(\cite{HRS1996})
  The category \(\cA^{\dagger}\) is the heart of a bounded t-structure on \(\cD\).
\end{proposition}

Now let \(\sigma = (\cA, Z)\) be a weak stability condition and \(\mu \in \RR\). We define the pair \((\cT_{\sigma}^{\mu}, \cF_{\sigma}^{\mu})\) of subcategories of \(\cA\) with
\begin{align*}
  \cT_{\sigma}^{\mu} &= \{E: \phi^{-}_{\sigma}(F) > \mu\} \\
  &= \langle E: E \text{ is } \sigma\text{-semistable with } \mu_{\sigma}(E) > \mu \rangle
\end{align*}
\begin{align*}
  \cF_{\sigma}^{\mu} &= \{E: \phi^{+}_{\sigma}(F) \leq \mu\} \\
                     &= \langle E: E \text{ is } \sigma\text{-semistable with } \mu_{\sigma}(E) \leq \mu \rangle
\end{align*}
By the existence of HN filtration, \((\cT, \cF)\) is a torsion pair. We now recall the construction of geometric stability conditions on surfaces and threefolds. Let \(X\) be an \(n\)-dimensional projective smooth variety \(H\) an ample divisor on \(X\). We define the lattice \(\Lambda_{H}^{j} \cong \ZZ^{j + 1}\) to be the image of the \(v^{j}_{H}\):
\[v^{j}_{H}: K(X) \to \QQ^{j + 1}, E \to (H^{n} \ch_{0}(E), \cdots, H^{n -j} \ch_{j}(E)) \in \QQ^{j + 1}\]
for \(j = 0, \dots, n\).

Then \(\sigma_{0} := (\cA = \Coh(X), Z^{1}_{H}(E) = iH^{n} \ch_{0}(E) - H^{n - 1} \ch_{1}(E))\) is a weak stability condition with respect to \(\Lambda_{H}^{1}\), as we may choose \(Q = 0\) in \cref{lem:4}.

We write \(\Coh^{\beta}_{H}(X) = \langle \cF_{\sigma_{0}}^{\beta}[1], \cT_{\sigma_{0}}^{\beta}\rangle_{\ext}\) where \(\beta \in \RR\). We denote
\[\ch^{\beta}(E) = e^{- \beta H} \ch(E) = (\ch_{0}^{\beta}(E), \dots, \ch_{n}^{\beta}(E)) \in H^{*}(X, \RR)\]

\begin{lemma}(\cite[Proof of Lemma 3.2.4]{BMT2014})
  \(\Coh^{\beta}_{H}(X)\) is Noetherian.
\end{lemma}

We can construct a weak stability condition on this tilted heart.

\begin{proposition}(\cite{BMT2014})
  Given \(\alpha > 0\) and \(\beta \in \RR\), the pair \(\sigma_{\alpha, \beta}^{2} = (\Coh_{H}^{\beta}(X), Z^{2}_{\alpha, \beta})\) defines a weak stability condition on \(D^{b}(X)\) with respect to \(\Lambda_{H}^{2}\). Here
  \[Z^{2}_{\alpha, \beta}(E) := i H^{n - 1} \ch_{1}^{\beta}(E) + \frac{1}{6} \alpha^{2} H^{n} \ch_{0}^{\beta} - H^{n - 2} \ch_{2}^{\beta}(E)\]
  Furthermore, if \(\dim X = 2\), it is a stability condition.
\end{proposition}

The Noetherian property helps to prove the existence of HN filtrations. And another main ingredient of the above proposition is the Bogomolov inequality (see \cite{Langer2015}), that is, for \(Z_{H}\)-semistable sheaf \(E\), we have
\[H^{n - 2}(\ch_{1}^{2}(E) - 2 \ch_{0}(E) \ch_{2}(E)) \geq 0.\]
By the Hodge index theorem, we have
\[\Delta_{H} := (H^{n - 1} \ch_{1}(E))^{2} - 2 H^{n} \ch_{0}(E) \cdot H^{n - 2} \ch_{2}(E) \geq 0\]
which will be \(Q\) in~\cref{lem:4} to establish the support property.

If \(n = \dim X = 3\), we can further tilt \(\Coh^{\beta}_{H}(X)\) with slope \(Z^{2}_{\alpha, \beta}\) to obtain a new pair \(\sigma^{3}_{\alpha, \beta} = (\cA_{H}^{\alpha,\beta}, Z^{3}_{\alpha,\beta})\), where
\[\cA_{H}^{\alpha, \beta} := \langle \cF_{\sigma_{\alpha, \beta}}^{0}[1], \cT_{\sigma_{\alpha, \beta}}^{0} \rangle_{\ext}\]
\[Z^{3}_{\alpha, \beta}(E) = (-\ch_{3}^{\beta}(E) + \frac{\alpha^{2}}{2} H^{2}\ch_{1}^{\beta}(E)) + i(\alpha H \ch_{2}^{\beta}(E) - \frac{\alpha^{3}}{6} H^{3}\ch_{0}^{\beta}(E))\]

We have following two equivalent conjectures:
\begin{conjecture}(\cite[Conjecture 3.2.6]{BMT2014})
  \label{conj:1}
  The pair \((\cA_{H}^{\alpha, \beta}, Z^{3}_{\alpha, \beta})\) is a stability condition on \(D^{b}(X)\).
\end{conjecture}

\begin{conjecture}(\cite[Conjecture 3.2.7]{BMT2014})
  \label{conj:2}
  For any \(Z_{\alpha, \beta}^{2}\)-semistable object \(E \in \cB_{H}^{\beta}\) satisfying
  \[\Im Z_{\alpha, \beta}^{3} = 0\]
  we have
  \[\ch_{3}^{\beta}(E) < \frac{\alpha^{2}}{2} H^{2} \ch_{1}^{\beta}(E).\]
\end{conjecture}

\begin{proposition}(\cite[Corollary 5.2.4]{BMT2014})
  \cref{conj:1} and \cref{conj:2} are equivalent.
\end{proposition}

Let \(\omega = \alpha H\), we write \(\sigma^{3}_{\omega, \beta} = (\cA_{\omega}^{\beta}, Z_{\omega, \beta}^{3}) := \sigma^{3}_{\alpha, \beta}\), and if \(\beta = 0\), we simply denote it by \(\sigma_{\omega}^{3} = (\cA_{\omega}, Z_{\omega}^{3})\). The notation is well defined as it is independent of the choice of \(\alpha, H\).

\section{Review on inducing stability conditions}
\label{sec:induc-t-struct}

\subsection*{Semiorthogonal decompositions}

We quick review basic definitions and facts about semi-orthogonal decompositions. Let \(\cD\) be a triangulated category.

\begin{definition}
  A \textit{semiorthogonal decomposition} of \(\cD\) is a sequence of full admissible triangulated subcategories \(\cD_{1}, \cdots, \cD_{m} \subset \cD\) such that \(\Hom_{\cD}(E,F) = 0\) for \(E \in \cD_{i}, F \in \cD_{j}, i > j\) and \(\cD\) is the minimal triangulated category in \(\cD\) containing every \(\cD_{i}\).
\end{definition}

\begin{definition}
  An object \(E \in \cD\) is \textit{exceptional} if \(\Hom_{\cD}(E, E[p]) = 0\) for all \(p \neq 0\) and \(\Hom_{\cD}(E, E[p]) = \CC\). A collection \(\{E_{1}, \cdots, E_{m}\}\) is called an \textit{exceptional collection} if \(E_{i}\) are exceptional object and \(\Hom_{\cD}(E_{i}, E_{j}[p]) = 0\) for all \(p\) and \(i > j\).
\end{definition}

An exceptional collection \(\{E_{1}, \cdots, E_{m}\}\) in \(\cD\) gives a semiorthogonal decomposition
\[\cD = \langle \cD', E_{1}, \cdots, E_{m} \rangle\]
where we abuse the notation to denote the full subcategory in \(\cD\) generated by \(E_{i}\) also by notation \(E_{i}\), and
\[\cD' = \langle E_{1}, \cdots, E_{m} \rangle^{\perp} = \{G \in \cD : \Hom(E_{i}, G[p]) = 0, \text{ for all } p \text{ and } i\}\]
Similarly, it gives semiorthogonal decomposition
\[\cD = \langle E_{1}, \cdots, E_{m}, \cD'' \rangle\]
where \(\cD'' = {}^{\perp} \langle E_{1}, \cdots, E_{m} \rangle = \{G \in \cD: \Hom(G, E_{i}[p]) = 0, \text{ for all } p \text{ and } i\}\).

\begin{definition}
  A Serre functor on \(\cD\) is an equivalence \(S: \cD \to \cD\) such that for any two objects \(A, B \in \cA\), there exists an isomorphism
  \[\eta_{A, B}: \Hom(A, B) \xrightarrow{\sim} \Hom(B, S(A))^{*}\]
  functorial in \(A\) and \(B\).
\end{definition}

The following facts are a direct consequence of the definitions of semiorthogonal decomposition and Serre functor.
\begin{enumerate}
  \item If \(\cD = \langle \cD_{1}, \cD_{2} \rangle\), then \(K(\cD) = K(\cD_{1}) \oplus K(\cD_{2})\).
  \item If \(\cD = \langle \cD_{1}, \cD_{2} \rangle\) is a semiorthogonal decomposition, then
  \[\cD = \langle S(\cD_{2}), \cD_{1} \rangle = \langle \cD_{2}, S^{-1}(\cD_{1}) \rangle\]
  are also semiorthogonal decompositions.
\end{enumerate}

In this paper, we consider a special case of semiorthogonal decompositions i.e.\ those that follow from blow-up of a point. The following theorem is well-known.

\begin{theorem}(\cite{BO1995,Huybrechts2006})
  \label{thm:1}
  Let \(Y\) be a smooth subvariety of a smooth projective variety \(X\). Let \(f: \tilde{X} = \Bl_{Y}(X) \to X\) be the blow-up, denote by \(i: D \hookrightarrow \tilde{X}\) the inclusion of the exceptional divisor and \(\pi: E \to Y\) the projection map. Then
  \begin{enumerate}
    \item The functors \(Lf^{*}\) and \(\Phi_{k}\) are fully faithful, where
    \[\Phi_{k}(E) := i_{*} \circ (\cO_{D}(kD) \otimes E) \circ \pi^{*}: D^{b}(Y) \to D^{b}(\tilde{X})\]
    and
    \item We have a semiorthogonal decomposition
    \[D^{b}\Coh(\tilde{X}) = \langle Lf^{*}D^{b}\Coh(Y), \Phi_{0}(D^{b}\Coh(Y)), \cdots, \Phi_{c - 2}(D^{b}\Coh(Y)) \rangle\]
    where \(c\) is the codimension of \(Y\) in \(X\).
  \end{enumerate}

\end{theorem}

In particular, if \(Y\) is a point, then we have a semiorthogonal decomposition
\[D^{b}\Coh(\tilde{X}) = \langle Lf^{*}D^{b}\Coh(X), \cO_{E}, \cO_{E}(-E), \cdots, \cO_{E}(-n+2)E) \rangle\]
where \(n = \dim X\). Notice that \(\cO_{E}(-kE)\) are exceptional objects.

\subsection*{Inducing stability conditions}

Fixing a semiorthogonal decomposition \(\cD = \langle \cD_{1}, \cD_{2} \rangle\), we want to establish stability condition on \(\cD_{1}\) from stability condition on \(\cD\). The following results are obtained in \cite{BLM+2023}.

\begin{definition}
  A \textit{spanning class} of a triangulated category \(\cD\) is a set of objects \(\cG\) such that if \(F \in \cD\) satisfies \(\Hom(G, F[p]) = 0\) for all \(G \in \cG\) and all \(p \in \ZZ\), then \(F \cong 0\).
\end{definition}

The following lemma gives a sufficient condition for inducing t-structures on semiorthogonal components.

\begin{lemma}(\cite[Lemma 4.3]{BLM+2023})
  \label{lem:5}
  Let \(\cA \subset \cD\) be the heart of a bounded t-structure. Assume that the spanning class \(\cG\) of \(\cD_{2}\) satisfies \(\cG \subset \cA\), and \(\Hom(G, F[p]) = 0\) for all \(G \in \cG, F \in \cA\) and all \(p > 1\). Then \(\cA_{1} := \cD_{1} \cap \cA\) is the heart of a bounded t-structure on \(\cD_{1}\).
\end{lemma}

\begin{corollary}(\cite[Corollary 4.4]{BLM+2023})
  \label{cor:1}
  Let \(\cA \subset \cD\) be the heart of a bounded t-structure and \(\cG\) a spanning class for \(\cD_{2}\) such that \(\cG \subset \cA\) and \(S(\cG) \subset \cA_{1}\). Then \(\cA_{1} := \cA \cap \cD_{1}\) is the heart of a bounded t-structure on \(\cD_{1}\).
\end{corollary}

And with further conditions, we can induce Bridgeland stability condition on semiorthogonal components.

\begin{proposition}(\cite[Proposition 5.1]{BLM+2023})
  \label{prop:1}
  Let \(\sigma = (\cA, Z)\) be a weak stability condition on \(\cD = \langle \cD_{1}, \cD_{2} \rangle\), where \(\cD_{2} = \langle E_{1}, \cdots, E_{m} \rangle\) with the following properties for all \(i = 1, \cdots, m\):
  \begin{enumerate}
    \item \(E_{i} \in \cA\)
    \item \(\cA \cap \cD_{1}\) is a heart of a bounded t-structure on \(\cD_{1}\).
    \item \(Z(E_{i}) \neq 0\).
  \end{enumerate}
  Assume moreover that there are no object \(0 \neq F \in \cA_{1} := \cA \cap \cD_{1}\) with \(Z(F) = 0\). Then the pair \(\sigma_{1} = (\cA_{1}, Z_{1} := Z|_{K(\cD_{1})})\) is a stability condition on \(\cD_{1}\).
\end{proposition}

The following lemma in the proof of \cref{prop:1} is will be used independently later.
\begin{lemma}(\cite[Lemma 5.2]{BLM+2023})
  \label{lem:10}
  Let \((\cA, Z)\) be a weak stability condition. Let \(\cA \subset \cA\) be an abelian subcategory such that the inclusion functor is exact. Assume moreover that \(Z\) restricted to \(K(\cA_{1})\) is a stability function. Then Harder-Narasimhan filtration exists in \(\cA_{1}\) for the stability function \(Z\).
\end{lemma}

\section{Surface case}
\label{sec:surface-case}

Let \(f: X \to Y\) be a contraction of a single \((-1)\)-curve \(C \subset X\) to a point in \(Y\). We define the hearts of perverse t-structures
\[^{p}\Per(X/Y) \subset D^{b}\Coh(X)\]
for \(p \in \ZZ\), introduced in \cite{Bridgeland2002,Bergh2004}. Let \(\cC\) be the triangulated subcategory of \(D^{b}\Coh(X)\), defined by
\[\cC := \{E \in D^{b}\Coh(X) : R f_{*} E = 0\}\]
The standard t-structure on \(D^{b}\Coh(X)\) induces a t-structure on \(\cC\), \((\cC^{\leq 0}, \cC^{\geq 0})\). We define \(^{p}\Per(X/Y)\) to be
\[^{p}\Per(X/Y) := \{E \in D^{b}\Coh(X) : Rf_{*}E \in \Coh(Y), \Hom(\cC^{< p}, E) = \Hom(E, \cC^{> p}) = 0\}\]
And we denote \(\Per(X/Y) := ^{-1} \Per(X/Y)\).
Let \(\cE := \cO_{X} \oplus \cO_{X}(-C), \cA := f_{*} \cEnd(\cE)\)
\begin{theorem}(\cite{Bergh2004})
  We have a derived equivalence
  \[\Phi:= Rf_{*} R\cHom(\cE, *) : D^{b}\Coh(X) \simeq D^{b}\Coh(\cA)\]
  which restricts to an equivalence between \(\Per(X/Y)\) and \(\Coh(\cA)\).
\end{theorem}

We define
\[\Per_{\leq i}(X/Y) := \{E \in \Per(X/Y) : \Phi(E) \in \Coh_{\leq i}(\cA)\}\]
and write \(\Per_{0}(X/Y) := \Per_{\leq 0}(X/Y)\).

\begin{proposition}
  \label{prop:2}
  (\cite[Proposition 3.4]{Toda2013b})
  \begin{enumerate}
    \item The category \(\Per_{0}(X/Y)\) is a finite length abelian category with simple objects given by \[\cO_{C},\ \cO_{C}(-1)[1],\ \cO_{x},\ x \in X\backslash C.\]
    \item An object \(E \in D^{b}\Coh(X)\) is an object in \(\Per(X/Y)\) if and only if \(\cH^{i}(E) = 0\) for \(i \neq 0, -1\) and \(\Hom(\cH^{0}(E), \cO_{C}(-1)) = 0\) and
    \[R^{1}f_{*} \cH^{0}(E) = f_{*} \cH^{-1}(E) = 0\]
  \end{enumerate}
\end{proposition}

Let \(\omega\) be an ample divisor on \(Y\), then we define the slope function on \(\Per(X/Y)\):
\[\mu_{f^{*}\omega}(E) := \frac{\ch_{1}(E) \cdot f^{*}\omega}{\ch_{0}(E)}\]
if \(\ch_{0}(E) > 0\) and \(+\infty\) otherwise.

Let \((\cT_{f^{*} \omega}, \cF_{f^{*}\omega})\) be the pair of subcategories of \(\Per(X/Y)\) defined by
\[\cT_{f^{*} \omega} := \langle E: E \text{ is } \mu_{f^{*} \omega}\text{-semistable with } \mu_{f^{*}\omega}(E) > 0\rangle\]
\[\cF_{f^{*} \omega} := \langle E: E \text{ is } \mu_{f^{*} \omega}\text{-semistable with } \mu_{f^{*}\omega}(E) \leq 0\rangle\]
The pair \((\cT_{f^{*} \omega}, \cF_{f^{*}\omega})\) forms a torsion pair, and we define
\(\cB_{f^{*}\omega} \subset D^{b} \Coh(X)\) to be
\[\cB_{f^{*} \omega} := \langle \cF_{f^{*}\omega}[1], \cT_{f^{*} \omega} \rangle\]
and
\[Z_{f^{*} \omega}(E) = (- \ch_{2}(E) + \frac{\omega^{2}}{2} \ch_{0}(E)) + \sqrt{-1} \ch_{1}(E) \cdot f^{*} \omega\]
\begin{proposition}(\cite[Lemma 3.12,Proposition 3.13]{Toda2013b})
  \((Z_{f^{*}\omega}, \cB_{f^{*}\omega})\) is a Bridgeland stability condition.
\end{proposition}

Let \(D^{b}\Coh(X) = \langle D^{b}\Coh(Y), \cO_{C} \rangle\) be the standard semiorthogonal decomposition associated to blow-up of a point in a surface \cite{BO1995}. By the above construction we have that \(\cO_{c} \in \cB_{f^{*} \omega}(X/Y)\) and \(S(\cO_{C}) = \cO_{C}(-1)[2] \in \cB_{f^{*}\omega}[1]\), and \(Z(\cO_{C}) \neq 0\). Therefore, by applying \cref{prop:1}, we obtain a stability condition \((\cA_{1}, Z_{1})\) on \(D^{b}\Coh(Y)\).

\begin{theorem}
  \label{thm:2}
  \((\cA_{1}, Z_{1})\) is a geometric stability condition on \(Y\). More precisely,  \((\cA_{1}, Z_{1}) = \sigma^{2}_{\omega}\) in \cref{sec:revi-tilt-bridg}.
\end{theorem}

\begin{proof}
  Let \(E \in D^{b} \Coh(Y)\), then we have \(Z_{1}(E) = Z_{f^{*} \omega}(Lf^{*}E) = (- f^{*} \ch_{2}(E) + \frac{\omega^{2}}{2} \ch_{0}(E)) + \sqrt{-1} f^{*}(\ch_{1}(E) \cdot \omega)\) as in the \cref{sec:revi-tilt-bridg}. So it remains to show that the hearts coincide. First notice that \(\mu_{f^{*} \omega}(Lf^{*}E) = \mu_{\omega}(E)\). So the central charge function coincides.

  We first show that \(\Per(X/Y) \cap Lf^{*}D^{b}\Coh(Y) = Lf^{*}\Coh(Y)\). Notice that by applying
  \cref{cor:1}, we have that left-hand side is the heart of a bounded \(t\)-structure. Since the right-hand side is already a heart of \(Lf^{*}D^{b}\Coh(Y)\), it suffices to show that the right-hand side is contained in the left-hand side. By (2) of \cref{prop:2}, we only need to check three statements in the theorem. First, we have \(\Hom(Lf^{*}E, \cO_{C}(-1)) = 0\) and therefore \(\Hom(\cH^{0}(E), \cO_{C}(-1)) = 0\).

  Then we show that \(\cH^{i}(Lf^{*}E) = 0\) for \(i \neq 0, -1\). Since flat pullback is non-derived pullback and open embedding is flat, we can reduce to the case that \(Y\) is affine. Let \(Y = \Spec R\) and \(\mathfrak{m} \subset R\) corresponding to the blow up point. By smoothness, we have \(\mathfrak{m}\) generated by \(r_{1}, r_{2} \in R\). Then \({\rm Bl}_{pt}Y\) admits an affine covering given by \(X_{1} := \Spec R[\frac{r_{1}}{r_{2}}]\) and \(X_{2} :=\Spec R[\frac{r_{2}}{r_{1}}]\). Let \(\iota_{j}: X_{j} \to X\) be the open embedding. It suffice to show that for \(f_{j}:X_{j} \to Y\), \(\cH^{i}(Lf_{j}^{*}E) = \iota_{j}^{*} \cH^{i}(Lf^{*}E) = 0\) for \(i \neq 0, -1\). But we have \(Lf_{j}^{*}(E) = \widetilde{M} \otimes^{L}_{R}R[\frac{r_{j}}{r_{3 - j}}]\) where \(M\) is the \(R\)-module corresponding to coherent sheave \(E\). Also \(R[\frac{r_{j}}{r_{3 - j}}]\) admits two term free resolution over \(R\). Therefore, we have that \(\cH^{i}(Lf_{j}^{*}E) = 0\) for \(i \neq 0, -1\) and thus \(\cH^{i}(Lf^{*}E) = 0\) for \(i \neq 0, -1\).

  Also, we have that \(Rf_{*} \circ Lf^{*}E = E\). And for \(E \in D^{b} \Coh(X)\), we have spectral sequence \(E_{2}^{p,q} = R^{p}f_{*}(\cH^{q}(E)) \Rightarrow R^{p + q}f_{*}(E)\). Because \(R^{i}f_{*} \cH^{0}(E) = 0\) for \(i \geq 2\). So the spectral sequence degenerates at \(E_{2}\). So after we replace \(E\) by \(Lf^{*}E\), we have \(R^{1}f_{*} \cH^{0}(Lf^{*}E) = f_{*} \cH^{-1}(Lf^{*}E) = 0\).

  Then \cref{lem:6} shows that \(Rf_{*}E\) is \(\mu_{\omega}\)-semistable if \(E\) is \(\mu_{f^{*} \omega}\)-semistable. So the HN filtration of \(0 = E_{0} \subsetneq E_{1} \subsetneq \cdots \subsetneq E_{n} = Lf^{*}E \in \Per(X/Y)\) induces an HN filtration of \(0 = Rf_{*}E_{0} \subset \cdots \subset Rf_{*}E_{n} = E \in \Coh(Y)\) after removing some possibly identical terms. Therefore, we have \(\cT_{\omega}  = \cT_{f^{*} \omega} \cap Lf^{*}D^{b}(Y)\) and \(\cF_{\omega}  = \cF_{f^{*} \omega} \cap Lf^{*}D^{b}(Y)\). Then we apply \cref{lem:7} to finish the proof.
\end{proof}

\begin{lemma}
  \label{lem:6}
  Let \(\sigma = (\cA, Z)\) be a weak stability condition on \(\cD\) with Serre functor \(S\) and semiorthogonal decomposition \(\cD = \langle \cD_{1}, \cD_{2}\rangle\) such that the following conditions hold
  \begin{enumerate}
    \item \(\cA_{1} = \cD_{1} \cap \cA\) and \(\cA_{2} = S(\cD_{2}) \cap \cA\) are the hearts of bounded t-structures for \(\cD_{1}\) and \(S(\cD_{2})\) respectively,
    \item \(Z(E) = 0\) for every \(E \in \cA_{2}\), and
    \item if \(j^{!}\) is the right adjoint to the inclusion \(j_{*}: \cD_{1} \to \cD\), we have \(j^{!} \cA \subset \cA_{1}\).
  \end{enumerate}
  Then if \(E \in \cA\) is \(\sigma\)-semistable, then \(j^{!}E\) is \(\sigma'\)-semistable, where \(\sigma' := (\cA_{1}, Z|_{K(\cD_{1})})\).
\end{lemma}

\begin{remark}
  This \(\sigma'\) may not be a weak stability condition, because the support property may not hold for \(\sigma'\). However, we can still define semistability on \(\cA\) by slope function \(\mu_{Z|_{K(\cD_{1})}}\).
\end{remark}

\begin{proof}
  Recall if \(\cD = \langle \cD_{1}, \cD_{2}\rangle\) is a semiorthogonal decomposition, then \[\cD = \langle S(\cD_{2}), \cD_{1} \rangle\] is also a semiorthogonal decomposition. Let \(E \in \cA\) be \(\sigma\)-semistable. Let
  \[i_{*}: S(\cD_{2}) \to \cD\]
  be the inclusion and \(i^{*}\) its left adjoint. Notice that by condition (2) for every \(F \in S(\cD_{2})\), we have \(Z(F) = 0\). Then we have a distinguished triangle
  \[j_{*}j^{!} E \to E \to i_{*}i^{*}E \xrightarrow{ [1]}\]
  Let us denote \(G := i_{*}i^{*}E\). We have associated long exact sequence
  \[0 \to H^{-1}_{\cA}(G) \to j_{*}j^{!}E \to E \to H^{0}_{\cA}(G) \to 0\]
  where \(H^{i}_{\cA}(G) \subset \cA_{2}\). If \(j^{!}E\) is zero, we are done. If it is non-zero, we further argue by contradiction. If \(j^{!}E\) is not \(\sigma'\)-semistable, there exists \(0 \neq F \subset j^{!}E\) such that \(\mu_{Z}(F) > \mu_{Z}(j^{!}E)\).

  On the one hand, we have that \(Z(E) = Z(j_{*}j^{!}E) = Z|_{K(\cD_{1})}(j^{!}E)\).

  We have \(j_{*}F \cap H^{-1}_{\cA}(G) \subset H^{-1}_{\cA}(G)\) and short exact sequence in \(\cA\),
  \[0 \to j_{*}F \cap H^{-1}_{\cA}(G) \to H^{-1}_{\cA}(G) \to K \to 0\]
  Since \(Z(H_{\cA}^{-1}(G)) = 0\). We have \(Z(K) = Z(j_{*}F \cap H^{-1}_{\cA}(G)) = 0\). So we have that \(\mu_{Z}(j_{*}F / (j_{*}F \cap H_{\cA}^{-1}(G))) = \mu_{Z}(j_{*}F)\). Since \(j_{*}F \in \cD_{1}\) and \(H_{\cA}^{-1}(G) \in \cD_{2}\), we have that \(j_{*}F / (j_{*}F \cap H_{\cA}^{-1}(G)) \neq 0\).

  On the other hand, we have \(j_{*}F/(j_{*}F \cap H^{-1}_{\cA}(G)) \subset j_{*}j^{!}E/ H^{-1}_{\cA}(G) \subset E\). But by the above argument, we have shown that \(\mu_{Z}(j_{*}F/(j_{*}F \cap H^{-1}_{\cA}(G))) > \mu_{Z}(E)\) which contradicts the semistability of \(E\). So we conclude that \(j^{!}E\) must be \(\sigma'\)-semistable.
\end{proof}

\begin{lemma}
  \label{lem:7}
  Let \(\cD = \langle \cD_{1}, \cD_{2} \rangle\) semiorthogonal decomposition and \(\cA\) be the heart of a bounded t-structure on \(\cD\) with torsion pair \((\cT, \cF)\), \(\cA_{1}\) be the heart of a bounded t-structure on \(\cD_{1}\) with torsion pair \((\cT_{1}, \cF_{1})\). If the following conditions hold
  \begin{enumerate}
    \item \(\cD_{2} = \langle E_{1}, \cdots, E_{n} \rangle\) with \(E_{i}\) exceptional object satisfying \(E_{i} \in \cT\) and \(S(E_{i}) \in \cT[1]\).
    \item \(\cA_{1} = \cA \cap \cD_{1}\), and
    \item \(\cT_{1} = \cT \cap \cD_{1}\), \(\cF_{1} = \cF \cap \cD_{1}\).
  \end{enumerate}
  then we have \(\cA_{1}^{\dagger} = \langle \cF_{1}[1], \cT_{1} \rangle_{\ext} = \cA^{\dagger} \cap \cD_{1 }= \langle \cF[1], \cT\rangle_{\ext} \cap \cD_{1}\).
\end{lemma}

\begin{proof}
  By definition of \(\cA_{1}^{\dagger}\), we have for every \(E \in \cA_{1}^{\dagger}\) a distinguished triangle
  \[F_{1}[1] \to E \to T_{1} \xrightarrow{ [1]}\]
  with \(F_{1}[1] \in \cF[1]\) and \(T_{1} \in \cT\). Since we have \(\cT_{1} \subset \cT\) and \(\cF_{1} \subset \cF\), we have \(E \in \cA^{\dagger}\). And therefore \(\cA_{1}^{\dagger} \subset \cA^{\dagger} \cap \cD_{1}\). But by applying \cref{cor:1}, we know that \(\cA^{\dagger} \cap \cD_{1}\) is also a heart of a bounded t-structure. Therefore, by~\cref{lem:2} we must have \(\cA_{1}^{\dagger} = \cA^{\dagger} \cap \cD_{1}\).
\end{proof}

\begin{remark}
  While we were writing this paper, we noticed that our discussion on stability on surfaces has significant similarity with \cite{Chou2024}, where surfaces with ADE singularity are considered. However, in their paper, they descend weak stability condition on resolution to obtain stability condition on singular surface, while in our case we have truly stability condition on blow up surfaces and descend it to contracted surfaces. And our result can not be directly apply to singular surfaces because semiorthogonal decomposition in \cref{thm:1} fails.
\end{remark}
\section{Threefold case}
\label{sec:threefold-case-1}

Let \(Y\) be a smooth projective threefold, we consider following two cases.
\begin{enumerate}
  \item \(f: X \to Y\) blowing up of a point in \(Y\).
  \item \(f: X \to Y\) blowing up of a smooth curve \(C \in Y\).
\end{enumerate}

In both cases, we denote by \(D \subset X\) the exceptional divisor.

\begin{theorem}
  Let \(\cE = \cO_{X} \oplus \cO_{X}(-D) \oplus \cV\), and \(\cA = f_{*} \cEnd(\cE)\), we have the equivalence of derived categories
  \[\Phi := Rf_{*} R \cHom(\cE, *): D^{b}\Coh(X) \xrightarrow{\sim} D^{b}\Coh(\cA)\]
  where
  \begin{enumerate}
    \item \(\cV = \cO_{X}(-2D)\) in case (1).
    \item \(\cV = 0\) in case (2).
  \end{enumerate}
\end{theorem}
\begin{proof}
  This is essentially \cite[Theorem 4.5]{Toda2013b}. We compose the original one with derived equivalence of \(D^{b} \Coh X\),  \(\Psi: E \to E \otimes \cO(2D)\) in case (1), or with \(E \to E \otimes \cO(D)\) in case (2).
\end{proof}

From equivalence
\[\Phi := Rf_{*} R \cHom(\cE, *): D^{b}\Coh(X) \xrightarrow{\sim} D^{b}\Coh(\cA)\]
We define \(\Per(X/Y)\) to be the pullback of the standard heart of \(D^{b}\Coh(\cA)\), and define \[\Per_{\leq i}(X/Y) := \{E \in \Per(X/Y): \Phi(E) \in \Coh_{\leq i}(\cA)\}\]
and write \(\Per_{0}(X/Y) := \Per_{\leq 0}(X/Y)\). So our definition differs from one in \cite{Toda2013b} by a tensor product.

\begin{proposition}
  \label{prop:3}
  \begin{enumerate}
    \item In case (1),
    \[\Per_{0}(X/Y) = \langle i_{*} \cO_{\PP^{2}}(-1)[2], i_{*} \Omega_{\PP^{2}}(1)[1], i_{*}\cO_{\PP^{2}}, \cO_{x} : x \in X \backslash D \rangle_{\ext}\]
    where \(i: D \to X\) is the embedding. In particular, we have \(i_{*}\cO_{\PP^{2}}(1), i_{*}\cO_{\PP^{2}}(-2)[2] \in \Per_{0}(X/Y)\).
    \item In case (2),
    \[\Per_{0}(X/Y) = \langle \cO_{x}, \cO_{L_{y}}(-1)[1], \cO_{L_{y}} : x \in X \backslash D, y \in f(D) \rangle\]
    where \(L_{y} := f^{-1}(y) \cong \PP^{1}\).
  \end{enumerate}
\end{proposition}
\begin{proof}
  For case (1), the first sentence is the rephrase of \cite[Proposition 4.15 and 4.21]{Toda2013b}. And the last statement follows from distinguished triangle
  \[\cO_{\PP^{2}}^{\oplus 3} \to \cO_{\PP^{2}}(1) \to \Omega_{\PP^{2}}(1)[1] \to \]
  and
  \[\to \Omega_{\PP^{2}}(1)[1] \to \cO_{\PP^{2}}(-2)[2] \to \cO_{\PP^{2}}(-1)^{\oplus 3}[2] \to\]
  For case (2), the result is the rephrase of \cite[Proposition 4.16 and 4.21]{Toda2013b}.

\end{proof}

Let \(\omega\) be an ample divisor on \(Y\). We define \(f^{*} \omega\)-slope function on \(\Per(X/Y)\). For \(E \in \Per(X/Y)\), the \(f^{*} \omega\)-slope is defined by
\[\mu_{f^{*} \omega}(E) = \frac{\ch_{1}(E) \cdot (f^{*} \omega)^{2}}{\ch_{0}(E)}\]
if \(\ch_{0}(E) \neq 0\) and \(\infty\) if \(\ch_{0}(E) = 0\).

\begin{lemma}
  Any object in \(\Per(X/Y)\) admits a Harder-Narasimhan filtration with respect to \(\mu_{f^{*} \omega}\)-stability.
\end{lemma}
\begin{proof}
  Notice tensoring \(\cO(2D)\) or \(\cO(D)\) doesn't change the slope function. And therefore, it follows from \cite[Lemma 5.1]{Toda2013b}.
\end{proof}

We define the torsion pair \((\cT_{f^{*} \omega}, \cF_{f^{*} \omega})\) in \(\Per(X/Y)\) as
\[\cT_{f^{*} \omega} := \langle E: E \text{ is } \mu_{f^{*} \omega}\text{-semistable with } \mu_{f^{*}\omega}(E) > 0\rangle\]
\[\cF_{f^{*} \omega} := \langle E: E \text{ is } \mu_{f^{*} \omega}\text{-semistable with } \mu_{f^{*}\omega}(E) \leq 0\rangle\]
and define
\[\cB_{f^{*} \omega} = \langle \cF_{f^{*} \omega}[1], \cT_{f^{*} \omega} \rangle_{\ext}\]

\begin{lemma}
  The abelian category \(\cB_{f^{*} \omega}\) is noetherian.
\end{lemma}

\begin{proof}
  This is \cite[Lemma 5.2]{Toda2013b}.
\end{proof}

Let \(B = dD\), with \(d \in \QQ\) satisfying \(- \frac{\sqrt{6}}{3} < b < \frac{\sqrt{6}}{3}\) in case (1) or \(-\frac{1}{2} < b < \frac{1}{2}\) in case (2) (\cite[Lemma 5.3]{Toda2013b}, there is a little difference because our heart is tensoring \(\cO(2D)\) or \(\cO(D)\) of his) and
\[Z_{B, f^{*} \omega}(E) = (- \ch_{3}^{B}(E) + \frac{(f^{*} \omega)^{2}}{2} \ch_{1}(E)) + \sqrt{-1}(f^{*} \omega \ch_{2}^{B}(E) - \frac{\omega^{3}}{6} \ch_{0}(E))\]
Here, we use \(\ch_{i}^{B}(E) (f^{*} \omega)^{3 - i} = \ch_{i}(E) (f^{*} \omega)^{3 - i}\) for \(i = 0, 1\).
We can further define the slope function \(\nu_{B, f^{*} \omega}\) on \(\cB_{f^{*} \omega}\) to be
\[\nu_{B, f^{*} \omega}(E) = \frac{\Im Z_{B, f^{*} \omega}(E)}{\ch_{1}(E) \cdot (f^{*} \omega)^{2}}\]
if \(\ch_{1}(E) \cdot (f^{*} \omega)^{2} \neq 0\) and \(\nu_{B, f^{*} \omega}(E) = \infty\) otherwise. We can define the \(\nu_{B, f^{*} \omega}\)-(semi)stability with slope function \(\nu_{B, f^{*} \omega}\) on abelian category \(\cB_{f^{*} \omega}\). Then we have

\begin{lemma}
  Any object in \(\cB_{f^{*} \omega}\) admits a Harder-Narasimhan filtration with respect to \(\nu_{B, f^{*} \omega}\)-stability.
\end{lemma}
\begin{proof}
  This follows from \cite[Lemma 5.7]{Toda2013b}.
\end{proof}

Again, we consider the torsion pair \((\cT'_{B, f^{*} \omega}, \cF'_{B, f^{*} \omega})\) in \(\cB_{f^{*} \omega}\) where
\[\cT'_{B, f^{*} \omega} := \langle E: E \text{ is } \nu_{B, f^{*} \omega}\text{-semistable with } \nu_{B, f^{*}\omega}(E) > 0\rangle\]
\[\cF'_{B, f^{*} \omega} := \langle E: E \text{ is } \nu_{B, f^{*} \omega}\text{-semistable with } \nu_{B, f^{*}\omega}(E) \leq 0\rangle\]
And define
\[\cA_{B, f^{*} \omega} := \langle \cF'_{B, f^{*} \omega}[1], \cT'_{B, f^{*} \omega}\rangle\]
We denote by \(\cA_{f^{*} \omega} = \cA_{B = 0, f^{*} \omega}\) and \(Z_{f^{*} \omega} = Z_{B = 0, f^{*} \omega}\).

\begin{theorem}
  \label{thm:3}
  If \(Y\) satisfy the generalized Bogomolov inequality in~\cref{conj:2}, i.e.\ for \(\nu_{\omega}\)-semistable object \(E \in \cB_{\omega}\) satisfying \(\nu_{\omega}(E) = 0\), we have that \(\ch_{3}(E) < \frac{\omega^{2}}{2} \ch_{1}(E)\), then \(\sigma_{f^{*} \omega} = (\cA_{f^{*} \omega}, Z_{f^{*} \omega})\) is a Bridgeland stability condition over \(D^{b}(X)\).
\end{theorem}

The proof of the above theorem will be postponed to the end of the section.

Let
\[D^{b}\Coh(X) = \langle  Lf^{*}D^{b}\Coh(Y), \cO_{D}, \cO_{D}(1) \rangle\]
in case (1) and
\[D^{b}\Coh(X) = \langle Lf^{*}D^{b}\Coh(Y), \Phi_{0}D^{b}(C) \rangle\]
in case (2) be the standard semiorthogonal decompositions associated to blow-up in \cref{thm:1}. By above construction, we have that \(\cO_{D}, \cO_{D}(1) \in \cA_{B, f^{*} \omega}\) and \(S(\cO_{D}) = \cO_{D}(-2)[3], S(\cO_{D}(1)) = \cO_{D}(-1)[3] \in \cA_{B, f^{*} \omega}[1]\) in case (1) and therefore, by applying \cref{prop:1} to \(\sigma_{B = 0, f^{*} \omega}\), we obtain Bridgeland stability condition \((\cA', Z')\) on \(D^{b}\Coh(Y)\) in case (1). However, the condition \cref{lem:5} is not satisfied in case (2). In below, we give another sufficient condition to induce stability conditions on semiorthogonal decompositions different from~\cref{prop:1}.

\begin{lemma}
  \label{lem:11}
  \(\Phi_{0}\Coh(C) \subset \Per(X/Y)\) and \(\Phi_{-1} \Coh(C) \subset \Per(X/Y)[-1]\).
\end{lemma}

\begin{proof}
  For the first inclusion, it suffices to show that \(\Phi \circ \Phi_{0} \Coh(C) \subset \Coh(\cA)\). Let \(F \in \Coh(C)\), let \(\pi: D = \PP(\cN_{C/X})\to C\) be the projection and \(i_{D}: D \to X\), \(i_{C}: C \to Y\) the embeddings.
  \begin{align*}
    \Phi \circ \Phi_{0} F & = Rf_{*}R \cHom(\cE, i_{D *} \circ \pi^{*} F) \\
                          & = Rf_{*} R \cHom(\cO \oplus \cO(-D), i_{D *} \circ \pi^{*} F) \\
                          & = Rf_{*} (i_{D *} \circ \pi^{*}F \otimes (\cO \oplus \cO(D))) \\
                          & = Rf_{*} i_{D *} \circ (\pi^{*} F \oplus \pi^{*}F \otimes \cO(-1)) \\
                          & = i_{C *} \circ R \pi_{*} (\pi^{*}F \oplus \pi^{*} F \otimes \cO(-1)) \\
                          & = i_{C *} F \in \Coh(\cA)
  \end{align*}

  The second inclusion can be deduced similarly.
\end{proof}

The following lemma can induce the heart of a bounded t-structure on semiorthogonal components either in case (1) or in case (2).
\begin{lemma}
  \label{lem:8}
  Let \(\cD = \langle \cD_{1}, \cD_{2} \rangle\) be a semiorthogonal decomposition and \(\cA\) be a heart of a bounded t-structure and \(j_{*} : \cD_{1} \to \cD\) the inclusion map with \(j^{!}: \cD \to \cD_{1}\) its right adjoint. If \(j^{!}\cA \subset \cA \cap \cD_{1}\) and \(j^{!}j_{*} = \id_{\cD_{1}}\), then \(\cA \cap \cD_{1}\) is a heart of a bounded t-structure of \(\cD_{1}\).
\end{lemma}

\begin{proof}
  We check \(\cA \cap \cD_{1}\) satisfy the condition of~\cref{lem:1}. The condition (1) follows from \(\cA\) being the heart of a bounded t-structure. For condition (2), since \(\cA\) is a heart of a bounded t-structure, by~\cref{lem:1}, we have for every \(j_{*}E \in \cD\), there exists a sequence of morphisms
  \[0 = E_{0} \xrightarrow{\phi_{1}} E_{1} \to \cdots \xrightarrow{\phi_{m}} E_{m} = j_{*}E\]
  such that \(\cone(\phi_{i}) \in \cA[k_{i}]\) for \(k_{1} > k_{2} > \cdots > k_{m}\). Now apply \(j^{!}\) to above sequence, we have
  \[0 = j^{!}E_{0} \xrightarrow{\psi_{1}} j^{!}E_{1} \to \cdots \xrightarrow{\phi_{m}} j^{!}E_{m} = j^{!}j_{*}E = E\]
  such that \(\cone(\phi_{i}) \in j^{!} \cA[k_{i}] \subset (\cA \cap \cD_{1})[k_{i}]\). Therefore after removing some possible identical terms in the sequence, we show that \(\cA \cap \cD_{1}\) satisfy condition (2) of \cref{lem:1}.
\end{proof}

\begin{proposition}
  \label{prop:5}
  Let \(\sigma = (\cA, Z)\) be a weak stability condition on \(\cD = \langle \cD_{1}, \cD_{2} \rangle\), where \(j_{*} : \cD_{1} \to \cD\) the inclusion map with \(j^{!}: \cD \to \cD_{1}\) its right adjoint, with the following properties
  \begin{enumerate}
    \item \(j^{!}\cA \subset \cA_{1} := \cA \cap \cD_{1}\) and \(j^{!}j_{*} = \id_{\cD_{1}}\).
    \item \(\cA_{2} = \cA \cap \cD_{2}\) is the heart of a bounded t-structure of \(\cD_{2}\).
    \item \(S(\cA_{2})[k] \subset \cA\) for some fixed \(k\).
    \item \(Z(E) \leq 0\) for \(E \in S(\cA_{2})[k]\) and equality holds if and only if \(E = 0\).
  \end{enumerate}
  Assume moreover that there are no object \(0 \neq F \in \cA_{1}\) with \(Z(F) = 0\). Then the pair \(\sigma_{1} = (\cA_{1}, Z_{1} := Z|_{K(\cD_{1})})\) is a stability condition on \(\cD_{1}\).
\end{proposition}

\begin{proof}
  By~\cref{lem:8}, \(\cA_{1}\) is the heart of a bounded t-structure. The existence of HN filtration follows from~\cref{lem:10}. So it remains to show that the support condition holds for \(\sigma_{1}\). Let \(E \in \cA_{1}\) semistable with respect to \(Z_{1}\).

  We claim that \(j_{*}E\) is semistable with respect to \(Z\). If not, we have surjection \(j_{*}E \twoheadrightarrow B\) where \(0 \neq B \subset \cA\) and \(\mu_{Z}(B) > \mu_{Z}(j_{*}E)\). Applying \(j^{!}\) to above morphism, we have that \(E \twoheadrightarrow j^{!}B\) and by semistability of \(E\), we have \(\mu_{Z_{1}}(E) \leq \mu_{Z_{1}}(j^{!}B)\). On the other hand, from semiorthogonal decomposition \(\cD = \langle S(\cD_{2}), \cD_{1} \rangle\), we have
  \[j_{*}j^{!}B \to B \to C \xrightarrow{ [1]}\]
  where \(C \in S(\cD_{2})\). And therefore, we have long exact sequence in \(\cA\):
  \[0 \to H_{\cA}^{-1}(C) \to j_{*}j^{!}B \to B \to H_{\cA}^{0}(C) \to 0\]
  By (2) and (3), we have that \(H_{\cA}^{0}(C) \in S(\cA_{2})[k]\). And the composition of surjection \(j_{*}E \to B \to H_{\cA}^{0}(C)\) is zero by semiorthogonal decomposition. Therefore, we have \(H_{\cA}^{0}(C) = 0\). Since \(Z(H_{\cA}^{-1}(C)) \leq 0\), we have \(\mu_{Z_{1}}(j^{!}B) = \mu_{Z}(j_{*}j^{!}(B)) \geq \mu_{Z}(B) > \mu_{Z}(j_{*}E) = \mu_{Z_{1}}(E)\), which is a contradiction and finish the proof of the claim.

  Now, the support property for \(\sigma_{1}\) follows from the support property for \(\sigma\).
\end{proof}

The following is the replacement for \cref{lem:7}.
\begin{lemma}
  \label{lem:3}
  Let \(\cD = \langle \cD_{1}, \cD_{2} \rangle\) semiorthogonal decomposition and \(\cA\) be the heart of a bounded t-structure on \(\cD\) with torsion pair \((\cT, \cF)\), \(\cA_{1}\) be the heart of a bounded t-structure on \(\cD_{1}\) with the torsion pair \((\cT_{1}, \cF_{1})\). If the following conditions hold:
  \begin{enumerate}
    \item \(\cA_{1} = \cA \cap \cD_{1}\), \(j^{!} \cA \subset \cA_{1}\) and \(j^{!}j_{*} = \id_{\cD_{1}}\),
    \item \(\cT_{1} = \cT \cap \cD_{1}\), \(\cF_{1} = \cF \cap \cD_{1}\), \(j^{!} \cT \subset \cT_{1}\), and \(j^{!} \cF \subset \cF_{1}\),
  \end{enumerate}
  Then we have \(\cA_{1}^{\dagger} = \langle \cF_{1}[1], \cT_{1} \rangle_{\ext} = \cA^{\dagger} \cap \cD_{1 }= \langle \cF[1], \cT\rangle_{\ext} \cap \cD_{1}\).
\end{lemma}

\begin{proof}
  By definition of \(\cA_{1}^{\dagger}\), we have every \(E \in \cA_{1}^{\dagger}\) fit into a distinguished triangle
  \[F_{1}[1] \to E \to T_{1} \xrightarrow{ [1]}\]
  with \(F_{1} \in \cF_{1} \subset \cF\) and \(T_{1} \in \cT_{1} \subset \cT\). So we have \(E \in \cA^{\dagger} \cap \cD_{1}\). So we have that \(\cA_{1}^{\dagger} \subset \cA^{\dagger} \cap \cD_{1}\). On the other hand for \(F \in \cA^{\dagger}\), we have distinguished triangle
  \[F[1] \to E \to T \xrightarrow{ [1]}\]
  where \(F \in \cF\) and \(T \in \cT\). Apply \(j^{!}\) to above triangle, we have
  \[j^{!}F_{1}[1] \to j^{!}E \to j^{!}T_{1}\]
  Since \(j^{!} \cT \subset \cT_{1}\) and \(j^{!} \cF \subset \cF_{1}\), we have that \(j^{!} \cA^{\dagger} \subset \cA_{1}^{\dagger}\). By \cref{lem:8}, we have that \(\cA^{\dagger} \cap \cD_{1}\) is the heart of a bounded t-structure. Now after applying~\cref{lem:2}, we have that \(\cA_{1}^{\dagger} = \cA^{\dagger} \cap \cD_{1}\).
\end{proof}






\begin{theorem}
  \label{thm:4}
  \((\cA', Z') = (\cA_{f^{*} \omega} \cap Lf^{*}D^{b}\Coh(Y), Z_{f^{*} \omega}|_{K(D^{b} \Coh(Y))}) = \sigma_{\omega}^{3}\) defined in \cref{sec:revi-tilt-bridg}. Moreover, \((\cA', Z')\) is a Bridgeland stability condition on \(D^{b}\Coh(Y)\) if \((\cA_{f^{*} \omega}, Z_{f^{*} \omega})\) is a Bridgeland stability condition on \(D^{b}\Coh(X)\).
\end{theorem}

\begin{proof}
  The equality \(Z_{B, f^{*} \omega}|_{K(D^{b}(Y))} = Z_{\omega, B = 0}\), follows from \(f^{*}\omega \cdot f^{*}D = f^{*} (\omega \cdot D)\) for every \(D \in H^{*}(Y)\).

  Now we prove that \(\cA' = \cA_{\omega, B = 0}\). We first show that \(Lf^{*} \Coh(Y) \subset \Per(X/Y)\) which is equivalent to show that \(\Phi(Lf^{*}E) \in \Coh(\cA)\), which follows from that
  \[R \cHom(\cE, Lf^{*}E) = \oplus_{n = 0}^{k}R \cHom(\cO_{X}(-nD), Lf^{*}E) = \oplus_{n = 0}^{k} Lf^{*}E \otimes \cO_{X}(nD)\]
  and where \(k = 2\) in case (1) and \(k = 1\) in case (2),
  \[Rf_{*}(Lf^{*}E \otimes \cO_{X}(nD)) = E \otimes Rf_{*}(\cO_{X}(nD)) = E \otimes f_{*} \cO_{X}(nD)\]
  for \(0 \leq n \leq k\).

  We claim \(\cB_{f^{*} \omega} \cap Lf^{*}D^{b}(Y)  = \cB_{\omega, B = 0}\). The procedure is similar to surface case. Let \(E\) be \(\mu_{f^{*} \omega}\)-semistable. Then by~\cref{lem:6} \(Rf_{*}E\) is \(\mu\)-semistable. Therefore HN filtration of \(Lf^{*}E \in \Per(X/Y)\) induces HN filtration of \(E \in \Coh(Y)\). So we have \(\cT_{f^{*} \omega} \cap Lf^{*}D^{b}(Y) = Lf^{*} \cT_{\omega}\), \(\cF_{f^{*} \omega} \cap Lf^{*} D^{b}(Y) = Lf^{*} \cF_{\omega}\). And now we apply~\cref{lem:3} to finish the proof of the claim.

  To show \(\cA' = \cA_{\omega, B = 0}\), we just repeat the above argument.

  The ``moreover'' part of the theorem follows from~\cref{prop:5} where the conditions (2) and (3) of case (2) in the proposition are obtained from \cref{lem:11}.
\end{proof}

Now we start the proof of the~\cref{thm:3}. We first recall the sufficient condition proposed by Toda.
\begin{proposition}(\cite[Proposition 5.14]{Toda2013b})
  \label{prop:4}
  If for any \(\nu_{f^{*} \omega}\)-semistable object \(E \in \cB_{f^{*} \omega}\) with \(\nu_{f^{*} \omega} = 0\), we have the inequality
  \[\ch_{3}(E) < \frac{(f^{*} \omega)^{2}}{2} \ch_{1}(E)\]
  Then \((Z_{f^{*} \omega}, \cA_{f^{*} \omega})\) is a stability condition.
\end{proposition}

We then prove the following lemma.

\begin{lemma}
  \label{lem:9}
  \cref{conj:2} holds for \(\omega, B = 0\), then for any \(\nu_{f^{*} \omega}\)-semistable object \(E \in \cB_{f^{*} \omega}\) with \(\nu_{f^{*} \omega} = 0\), we have the inequality
  \[\ch_{3}(E)< \frac{(f^{*} \omega)}{2} \ch_{1}(E)\]
\end{lemma}

\begin{proof}
  By linearity and Jordan-Hölder filtration, it is sufficient to consider stable object \(E\) with \(\nu_{f^{*} \omega} = 0\).

  We again use the semiorthogonal decomposition
  \[D^{b}\Coh(X) = \langle S(\cD_{2}), D^{b}\Coh(Y) \rangle\]
  where \(\cD_{2} = \langle \cO_{D}, \cO_{D}(1) \rangle\) in case (1) and \(\Phi_{0}(D^{b} \Coh(C))\) in case (2). From the decomposition, we obtain the distinguished triangle
  \[Lf^{*}Rf_{*}E \to E \to G \xrightarrow{ [1]}\]
  and long exact sequence
  \[0 \to H_{\cB_{f^{*}\omega}}^{-1}(G) \to Lf^{*}Rf_{*}E \to E \to H_{\cB_{f^{*} \omega}}^{0}(G) \to 0\]
  Notice that for every \(F \in S(\cD_{2})\), \((f^{*} \omega)^{3 - i} \ch_{i}(F) = 0\) for \(i \leq 2\). On the other hand, we know \(S(\cA_{2})[k] \in \cB_{f^{*} \omega}\) for some \(k\), where \(\cA_{2} = \langle \cO_{D}, \cO_{D}(1) \rangle_{\ext}\) in case (1) and \(\cA_{2} = \Coh(C)\) in case (2). Therefore \(H_{\cB_{f^{*} \omega}}^{i}(G) \in S(\cA_{2})[k]\).

  As in proof of \cref{thm:4}, we know that \(Rf_{*}E\) is semistable with slope \(0\) in \(\cB_{\omega}\). Therefore, we have
  \[\ch_{3}(Lf^{*}Rf_{*}E) = \ch_{3}(Rf_{*}E) < \frac{\omega^{2}}{2}\ch_{1}(Rf_{*}E) = \frac{(f^{*}\omega)^{2}}{2}\ch_{1}(Lf^{*}Rf_{*}E)\]
  By short exact sequence above, we have that injection \(H := Lf^{*}Rf_{*}E/H^{-1}_{\cB_{f^{*} \omega}}(G) \to E\). If \(H = 0\), then we have \(E = H_{\cB_{f^{*} \omega}}^{0}(G)\) which is impossible to have \(\nu_{f^{*}\omega}(E)= 0\). So \(H \neq 0\) and \(\nu_{f^{*} \omega}(H) = \nu_{f^{*} \omega}(Lf^{*}Rf_{*} E) = 0\). Therefore, the maximal destabilizing subsheaf \(H_{1}\) of \(H\) has \(\nu_{f^{*} \omega}(H_{1}) \geq 0\). But since \(E\) is stable with \(\nu_{f^{*} \omega}(E) = 0\). We must have \(E = H_{1} = H\). So we have
  \begin{align*}
    \ch_{3}(E) &= \ch_{3}(Lf^{*}Rf_{*}(E)) - \ch_{3}(H_{\cB_{f^{*}\omega}}^{-1}(G)) \\
               & \leq \ch_{3}(Lf^{*}Rf_{*}(E)) \\
               & < \frac{(f^{*}\omega)^{2}}{2} \ch_{1}(Lf^{*}Rf_{*}E) \\
               & = \frac{(f^{*}\omega)^{2}}{2} \ch_{1}(E)
  \end{align*}
  The first inequality follows from the fact that \(- \ch_{3}(H_{\cB_{f^{*} \omega}}^{-1}(G)) \leq 0\).
\end{proof}

\begin{proof}[Proof of \cref{thm:3}]
  The result follows from \cref{lem:9} and \cref{prop:4}.
\end{proof}

\begin{corollary}
  \((\cA_{f^{*} \omega}, Z_{f^{*} \omega})\) is a stability condition if and only if \((\cA_{\omega}, Z_{\omega})\) is.
\end{corollary}

\begin{proof}
  This combines~\cref{thm:3} and~\cref{thm:4}.
\end{proof}

\begin{remark}
  As pointed out in \cite[Remark 5.10]{Toda2013b}, for \(0 < \epsilon \ll 1\), \(f^{*} \omega - \epsilon D\) is ample. So we would like to see if \((\cA_{f^{*} \omega}, Z_{f^{*} \omega})\) is a limit of geometric stability condition \((\cA_{f^{*} \omega - \epsilon D}, Z_{f^{*} \omega - \epsilon D})\). However, there are some gaps to be filled. The main problem is that we don't know any case where the \((\cA_{f^{*} \omega - \epsilon D}, Z_{f^{*} \omega - \epsilon D})\) is a stability condition. In \cite{BMS+2017}, the stability conditions in the case of blowing up of a point in \(\PP^{3}\) and blowing up of a line in \(\PP^{3}\) are considered. However, in the former case, the authors only prove some variant of the generalized Bogomolov inequality in with \(H' = 2 f^{*}H - D\), where \(H\) is the canonical polarization on \(\PP^{3}\). In the latter case, the generalized Bogomolov inequality was shown only for \(H' = a f^{*} H - b D\), where \(a, b \geq 0\) and \(a \leq b\), which also don't satisfy our needs.
\end{remark}

\begin{remark}
  We expect the similar situation occurs in higher dimensional case. However, there are no general result about the (geometric) stability conditions in higher dimensional case. So we limited ourselves to surface and threefold case.
\end{remark}

\printbibliography

\end{document}